\documentclass[preprint, 10p]{elsarticle}
\usepackage[utf8]{inputenc}
\usepackage[left=2cm,right=2cm,top=2cm,bottom=2cm]{geometry}

\usepackage{lineno,hyperref}
\modulolinenumbers[5]

\usepackage{amsmath}
\usepackage{amsthm}
\usepackage{amsfonts}
\usepackage{amssymb}
\usepackage{array}
\usepackage{caption}
\usepackage{enumitem}
\usepackage{framed}
\usepackage{graphicx}
\usepackage{pdfpages}
\usepackage{subcaption}
\usepackage{wrapfig}
\usepackage{xcolor}

\newtheorem{thm}{Theorem}
\newtheorem{lem}[thm]{Lemma}
\newtheorem{cor}[thm]{Corollary}
\newdefinition{rmk}[thm]{Remark}
\newdefinition{defn}[thm]{Definition}
\newdefinition{exmp}[thm]{Example}

\usepackage{multirow}

\usepackage{natbib}
\usepackage{etoolbox}
\apptocmd\normalsize{%
 \abovedisplayskip=4pt
 \abovedisplayshortskip=3pt
 \belowdisplayskip=4pt
 \belowdisplayshortskip=6pt
}{}{}

\usepackage{enumitem}
\setlist{topsep=5pt,itemsep=2pt} 

\begin{document}

\begin{frontmatter}
\title{Heterogeneous carrying capacities and global extinction in metapopulations}

\author[1]{Jakub Hesoun}
\ead{hesounj@kma.zcu.cz}

\author[1]{Petr Stehl\'{i}k\corref{cor1}}
\ead{pstehlik@kma.zcu.cz}

\cortext[cor1]{Corresponding author}

\address[1]{Department of Mathematics and NTIS, University of West Bohemia, Univerzitn\'{i} 8, 301~00 Pilsen, Czech Republic}

\begin{abstract}
In this paper we consider a simple two patch reaction diffusion model with strong Allee effect, sufficiently distinct carrying capacities, similar reaction strengths, and strong diffusion. In the homogeneous case, i.e., in in the case of equal or similar capacities and reaction strengths, it is well known that the number of stationary solutions ranges from three (strong diffusion) to nine (weak diffusion). We provide sufficient conditions which includes the diffusion strength and reaction parameters that ensure that the extinction point is the unique and globally asymptotically stable equilibrium in the case of heterogeneous capacities. For the sake of robustness we consider several bistable reaction functions, compare our analytical result with numerical simulations, and conclude the paper with a short discussion on global extinction literature (which has provided mostly numerical results so far), as well as other related phenomena, e.g., fragmentation, the perfect mixing paradox, and the natural form the reaction diffusion patch models.


\end{abstract}


\end{frontmatter}

\section{Introduction}
\noindent In this paper we consider a simple two-patch bistable reaction-diffusion model with heterogeneous capacities $k_1\neq k_2$

\begin{equation}\label{e:main}
    \begin{cases}
    \displaystyle x_1' = D(x_2-x_1) + \lambda_1 x_1 \left(1-\frac{x_1}{k_1} \right) \left(\frac{x_1}{k_1}-\frac{1}{2} \right),\\
    \displaystyle x_2' = D(x_1-x_2) + \lambda_2 x_2 \left(1-\frac{x_2}{k_2} \right) \left(\frac{x_2}{k_2}-\frac{1}{2} \right),    
    \end{cases}
\end{equation}
where $D>0$ represents the diffusion parameter and $\lambda_1,\lambda_2>0$ heterogeneous reaction strengths. Without loss of generality we assume that $k_2\leq k_1$. 

The system \eqref{e:main} has been extensively studied in the symmetric case $k_1=k_2$ and $\lambda_1=\lambda_2$. It was shown that there exist 9 stationary solutions (6 heterogeneous and 3 homogeneous)  for sufficiently small $D$ and three homogeneous solutions (two stable ones and one unstable) once $D$ is large enough, \cite{Gruntfest1997, Gyllenberg1999, Stehlik2017}. This result has deep consequences both for the theory of metapopulations \cite{Hanski1998} but also for patterns of discrete-space dynamical systems \cite{Hupkes2019c}. 

In contrast to this homogeneous case, our main result provides sufficient conditions on the heterogeneous capacities $k_i$, and reaction strengths $\lambda_i$ which ensure that for sufficiently strong diffusion, there is a unique equilibrium -- the origin which corresponds to the extinction of both populations, see also Fig.~\ref{fig:numerics}.
\begin{thm}\label{t:main}
Assume that $\max\{\lambda_1,\lambda_2\}<4D$, $2k_2<k_1$, and
\begin{equation}\label{e:ineq:l1l2}
    \max \left\{ \frac{2 \sqrt{3} k_1^2}{9 \left( k_1-k_2 \right) \left( 2 k_1-k_2 \right)} , \, \frac{\sqrt{3} k_1^2}{18 \left(k_1-2 k_2\right) \left( k_1-k_2\right)} \right\} < \frac{\lambda_1}{\lambda_2} < \frac{9 \left(k_1-2 k_2 \right) \left( k_1-k_2\right)}{2 \sqrt{3} k_2^2}.
\end{equation}
Then the system \eqref{e:main} has the unique stationary solution $(x_1^*,y_1^*)=(0,0)$ which is globally asymptotically stable.
\end{thm}

In other words, the sufficient condition from Thm.~\ref{t:main} requires that the first patch must be at least twice as large, diffusion strong enough, and the reaction strengths sufficiently close (inequalities \eqref{e:ineq:l1l2}). Comparable conditions have been obtained for similar models numerically \cite{Althagafi2021, Amarasekare1998, Amarasekare1998b}. 

In Sec.~\ref{sec:prelim} we reduce the number of parameters in the model~\eqref{e:main} and provide some straightforward results on the number of stationary solutions for similar capacities. In Sec.~\ref{sec:main} we prove Thm.~\ref{t:main} and in Sec.~\ref{sec:generalization} we generalize Thm.~\ref{t:main} to general bistabilities with arbitrary viability constants and provide explicit necessary and sufficient condition for sawtooth caricature. We conclude with a brief discussion on the optimality of our results and literature in Sec.~\ref{sec:discussion}.


\section{Preliminaries}\label{sec:prelim}
We first observe, that for small $D>0$ we have 9 stationary solutions.
\begin{thm}\label{t:smalld}
If $D>0$ is sufficiently small then the problem \eqref{e:main} has 9 nonnegative stationary solutions.
\end{thm}

\begin{proof}
    The results follows by a standard observation that for $D=0$ there are 9 stationary solutions and the application of the implicit function theorem which is enabled by nonzero derivatives of the cubic $f$ in $0,1/2,1$. See, e.g., \cite{Stehlik2017} for more details.
\end{proof}

Using $ x = x_1/k_1, y = x_2/k_2, \tau =  D t $, we reduce the number of parameters in~\eqref{e:main} and consider the system
\begin{equation}\label{e:balanced}
    \begin{cases}
    \displaystyle x' = \gamma y-x + \alpha f(x),\\
    \displaystyle y' = \frac{1}{\gamma}x - y + \beta f(y).    
    \end{cases}
\end{equation}
where $f(s):=s(1-s)(s-1/2)$ and
\begin{equation}\label{e:parameters}
    \gamma = \frac{k_2}{k_1}\in(0,1],\quad \alpha= \frac{\lambda_1}{D}>0,\quad \beta=\frac{\lambda_2}{D}>0.
\end{equation}



\begin{figure}
\begin{minipage}[c]{0.45\linewidth}
\includegraphics[width=\linewidth]{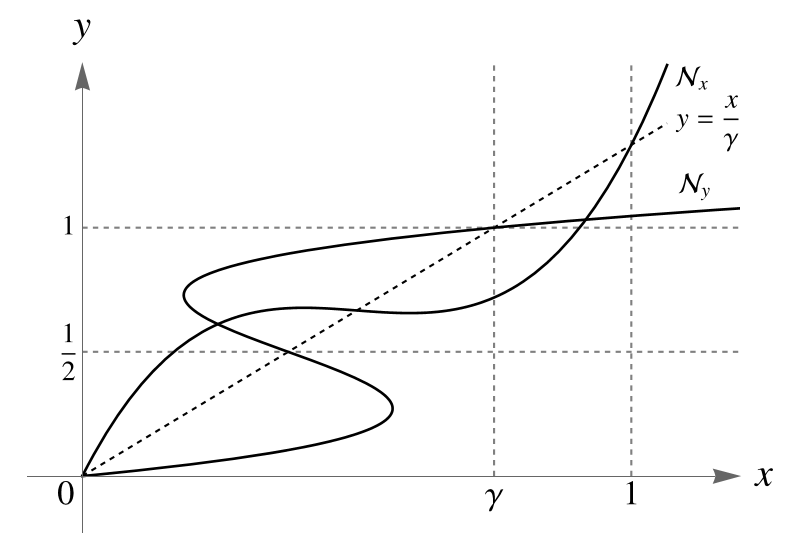}
\caption{Illustration of Lemma~\ref{lem:gamma:gr:onehalf}. Nullclines $\mathcal{N}_x$ and $\mathcal{N}_y$ from \eqref{e:nullclines} for $\alpha=6$, $\beta=8$\label{fig:nullclines}}
\end{minipage}
\hfill
\begin{minipage}[c]{0.45\linewidth}
\includegraphics[width=\linewidth]{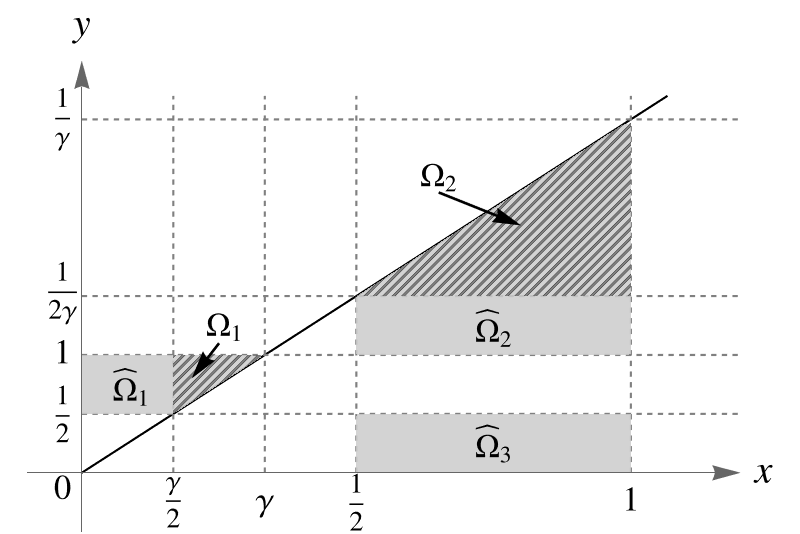}
\caption{A priori estimates $\widehat{\Omega}_i$ from Lemma~\ref{lem:3r} (light gray regions) and Lemma~\ref{lem:2r} (the hatched subsets $\Omega_i\subset \widehat{\Omega}_i$).}\label{fig:regions_omegas}
\end{minipage}
\end{figure}


Next, we show that there is always a nontrivial equilibrium once $\gamma\in[1/2,1)$, i.e., in the case in which the capacities do not differ too much.
\begin{lem}\label{lem:gamma:gr:onehalf}
Let $\alpha,\beta>0$, then the problem \eqref{e:balanced} has 
\begin{itemize}
    \item at least three nonnegative stationary solutions if $\gamma \in (1/2, 1)$,
    \item at least two nonnegative stationary solutions if $\gamma=1/2$.
\end{itemize}
\end{lem}
\begin{proof}
    For $\gamma = 1/2,$ we have a non-trivial solution $x^*=1/2, y^*=1$. For $\gamma<1/2$, stationary solutions of~\eqref{e:balanced} lie on intersections of continuous nullclines $\mathcal{N}_x$, $\mathcal{N}_y$ given by
    \begin{equation}\label{e:nullclines}
        \mathcal{N}_x: y=\frac{x}{\gamma} - \alpha \frac{f(x)}{\gamma}=:\nu_x(x),\quad \mathcal{N}_y: y=\frac{x}{\gamma} + \beta f(y).
    \end{equation}
    Comparing the positions of $\mathcal{N}_x$, $\mathcal{N}_y$ with respect to the line $\ell: y=\frac{x}{\gamma}$, we observe that 
    $\mathcal{N}_x$ lies above $\ell$ if $x\in(0,1/2)\cup(1,\infty)$ and $\mathcal{N}_y$ lies above the line $\ell$ if $y\in(1/2,1)$. Consequently, there are at least three intersections of $\mathcal{N}_x$ and $\mathcal{N}_y$. These three intersections $(x_i^*,y_i^*)$, $i=1,2,3$ satisfy (due to continuity) $x_1^*=y_1^*=0$, $y_2^*\in(1/2,1)$, and $x_3^*\in(\gamma,1)$. Since for $x> \gamma, \ \mathcal{N}_y$ is located above the line $y=1$, we have that $y_3^*>y_2^*$, see Fig.~\ref{fig:nullclines}.
\end{proof}

\section{Global extinction}\label{sec:main}
In this section we consider $\gamma\in(0,1/2)$ and prove our main result, Thm.~\ref{t:main}. We begin with auxiliary a priori estimates on nontrivial solutions, see Fig.~\ref{fig:regions_omegas}.
\begin{lem}\label{lem:3r}
Let $\alpha,\beta>0$, $\gamma \in (0, \, 1/2)$ and $(x^*,y^*)$ be a non-trivial stationary solution of \eqref{e:balanced}. Then $(x^*,y^*)\in \widehat{\Omega}_1 \cup \widehat{\Omega}_2 \cup \widehat{\Omega}_3 $, where

\begin{align}\label{e:3r}
    \begin{split}
        \widehat{\Omega}_1 &:= \left\{ (x, \, y) \in \mathbb{R}^2: x \in (0, \, \gamma), \, y \in \left( \frac{1}{2}, \,1 \right), \, y> \frac{x}{\gamma} \right\}, \\
        \widehat{\Omega}_2 &:= \left\{ (x, \, y) \in \mathbb{R}^2: x \in \left(\frac{1}{2}, \, 1\right), \, y \in \left( 1, \,\frac{x}{\gamma} \right) \right\},\\
        \widehat{\Omega}_3 &:= \left\{ (x, \, y) \in \mathbb{R}^2: x \in \left(\frac{1}{2}, \, 1 \right), \, y \in \left( 0, \, \frac{1}{2} \right) \right\}.
    \end{split}
\end{align}
\end{lem}
\begin{proof}
The stationary solution of \eqref{e:balanced} has to comply with~\eqref{e:nullclines}
which implies that
\begin{equation}\label{e:3r:sign}
\gamma y -x = -\alpha f(x) = \beta \gamma f(y).    
\end{equation}

Since $\alpha,\beta, \gamma>0$ the values of $f(x)$ and $f(y)$ must have opposite signs. We distinguish between two cases. First, let us assume that $f(x)<0$ and $f(y)>0$. This implies that $y\in(1/2,1)$ and $x<\gamma$, since
\[
0< \gamma y - x < \gamma -x 
\]
This defines the set $\widehat{\Omega}_1$.

Next, let $f(x)>0$ which leads to $x\in(1/2,1)$ and $f(y)<0$. We have two subcases. If $y\in(0,1/2)$ we get the set $\widehat{\Omega}_3$. Similarly we arrive to $\widehat{\Omega}_2$ by considering $y>1$ and observing that $y<x/\gamma$ from \eqref{e:3r:sign}.
\end{proof}


Once we consider small values of $\alpha, \beta$ (i.e., large diffusion $D$ in the original problem \eqref{e:main}) we can further reduce the regions $\widehat{\Omega}_i$ from~\eqref{e:3r}, see Fig.~\ref{fig:regions_omegas}.
\begin{lem}\label{lem:2r}
Let $\alpha,\beta\in(0,4)$, $\gamma \in (0, 1/2)$ and $(x^*,y^*)$ be a non-trivial stationary solution of \eqref{e:balanced}. Then $(x^*,y^*)\in \Omega_1 \cup \Omega_2$, where

\begin{align}\label{e:2r}
    \begin{split}
        \Omega_1 &= \left\{ (x, \, y) \in \mathbb{R}^2: x \in \left(\frac{\gamma}{2}, \, \gamma \right), \, y \in \left( \frac{x}{\gamma}, \,1 \right)\right\}, \\
        \Omega_2 &= \left\{ (x, \, y) \in \mathbb{R}^2: x \in \left(\frac{1}{2}, \, 1 \right), \, y \in \left( \frac{1}{2\gamma}, \, \frac{x}{\gamma}\right) \right\}.
    \end{split}
\end{align}
\end{lem}
\begin{proof}
The concavity of $f(s)$ on $(1/2,1)$ yields
\[
f(s)\leq \frac{1}{4} \left( s- \frac{1}{2} \right)
\]
and we can rewrite \eqref{e:nullclines} into 
\begin{align*}
    \displaystyle  y &= \frac{x-  \alpha f(x)}{\gamma} \geq \frac{x- \frac{\alpha}{4} \left(x - \frac{1}{2} \right) }{\gamma} = \frac{x\left(1 - \frac{\alpha}{4} \right) + \frac{\alpha}{8}}{\gamma} \quad \text{ if } x>
    \frac{1}{2},\\ 
\displaystyle  x &=  \gamma \left( y-  \beta f(y) \right) \geq \gamma \left(y- \frac{\beta}{4} \left(y - \frac{1}{2} \right) \right) = \gamma \left( y \left( 1 - \frac{\beta}{4} \right) + \frac{\beta}{8} \right) \quad \text{ if } y>
    \frac{1}{2}.
\end{align*}
If $\alpha,\beta\in(0,4)$ we have $1-\alpha/4>0$ and $1-\beta/4>0$ which implies
\begin{align*}
    \displaystyle  y &> \frac{1}{2\gamma} \quad \text{ if } x>
    \frac{1}{2} \quad \text{ and}\quad x >  \frac{\gamma}{2} \quad \text{ if } y>
    \frac{1}{2}.
\end{align*}
The former inequality eliminates $\widehat{\Omega}_3$ and reduces $\widehat{\Omega}_2$ into ${\Omega}_2$. The latter then reduces $\widehat{\Omega}_1$ into ${\Omega}_1$. 
\end{proof}

We next eliminate the possible location of stationary solution in $\Omega_1$.
\begin{lem}\label{lem:omega1}
Let $\alpha,\beta\in(0,4)$, $\gamma \in (0, 1/2)$ and assume
\begin{equation} \label{e:alpha:omega1}
\displaystyle \alpha > \beta\cdot  \max \left\{ \frac{2 \sqrt{3}}{9 \left( 1-\gamma \right) \left(2-\gamma \right)}, \, \frac{\sqrt{3}}{18 \left(1- 2 \gamma \right) \left(1- \gamma \right)} \right\},
\end{equation}
then the system \eqref{e:balanced} has no stationary solution located in $\Omega_1$ defined by \eqref{e:2r}.
\end{lem}
\begin{proof}
    The cubic $f(x)$ satisfies $f(x)\leq \frac{\sqrt{3}}{36}$ for $x\in(0,1)$. Consequently, $\mathcal{N}_y$ in~\eqref{e:nullclines} implies
    \begin{equation}\label{e:Ny:estimate}
        x=\gamma(y-\beta f(y)) \geq  \gamma\left(y-\beta\frac{\sqrt{3}}{36}\right)=:\Psi(y).
    \end{equation}
    Since the nullcline $\mathcal{N}_x$ in~\eqref{e:nullclines} satisfies
    \[
        \nu_x(\gamma/2)=\frac{1}{2}-\frac{1}{2} \alpha \left(1-\frac{\gamma }{2}\right) \left(\frac{\gamma }{2}-\frac{1}{2}\right),\quad \nu_x(\gamma)=1-\alpha (1-\gamma ) \left(\gamma -\frac{1}{2}\right)
    \]
    and is concave on $(\gamma/2,\gamma),$ we can only compare these two values with those of $\Psi^{-1}$, i.e.
    \[
        \Psi^{-1}(\gamma/2)=\frac{1}{2} + \beta\frac{\sqrt{3}}{36},\quad \Psi^{-1}(\gamma/2)=1 + \beta\frac{\sqrt{3}}{36}.
    \]
    The assumption \eqref{e:alpha:omega1} implies $\nu_x(\gamma/2)>\Psi^{-1}(\gamma/2)$ and $\nu_x(\gamma)>\Psi^{-1}(\gamma)$ which finishes the proof.    
\end{proof}
Note that for $\gamma\in(0,2/7)$ the first fraction in \eqref{e:alpha:omega1} is greater, for $\gamma=2/7$ both are equal, and for $\gamma\in(2/7,1/2)$ the latter is greater. In the same spirit, we eliminate the possible location of stationary solution in $\Omega_2$.
\begin{lem}\label{lem:omega2}
Let $\alpha,\beta\in(0,4)$, $\gamma \in (0, 1/2)$ and assume
\begin{equation} \label{e:alpha:omega2}
\displaystyle \alpha < \frac{9 \left(1-2\gamma\right) \left( 1- \gamma \right)}{2 \sqrt{3} \gamma^2} \beta.
\end{equation}
Then the system \eqref{e:balanced} has no stationary solution located in $\Omega_2$ defined by \eqref{e:2r}.
\end{lem}
\begin{proof}
    As in~\eqref{e:Ny:estimate} we have that the nullcline $\mathcal{N}_x$ lies above a line for $x\in(1/2,1)$, specifically, 
    \[
        y=\nu_x(x)=\frac{1}{\gamma}\left(x-\alpha f(x) \right) \geq \frac{1}{\gamma}\left(x-\alpha \frac{\sqrt{3}}{36} \right)=:\Phi(x).
    \]
    Clearly we have that $x_1=1/2+\alpha \sqrt{3}/{36}$ satisfies $\Phi(x_1)=1/(2\gamma)$. From the second equation in~\eqref{e:nullclines} we observe that the nullcline $\mathcal{N}_y$ enters $\Omega_2$ (i.e., crosses $y=1/(2\gamma)$) in
    \[
    x_2 = \frac{1}{2}+\frac{1}{2} \beta  \left(1-\frac{1}{2 \gamma }\right) \left(\frac{1}{2 \gamma }-\frac{1}{2}\right).
    \]
    From \eqref{e:alpha:omega2} we have that $x_2>x_1$ and the proof is finished by observing that $\mathcal{N}_y$ is concave for $x>x_2$.    
\end{proof}

Combining Lemmas~\ref{lem:omega1} and \ref{lem:omega2} we get the following statement.
\begin{cor}\label{cor:nonexistence}
Let $\alpha,\beta\in(0,4)$, $\gamma \in (0, 1/2)$ and assume
\begin{equation} \label{e:alpha:omegas}
\displaystyle\max \left\{ \frac{2 \sqrt{3}}{9 \left( 1-\gamma \right) \left(2-\gamma \right)}, \, \frac{\sqrt{3}}{18 \left(1- 2 \gamma \right) \left(1- \gamma \right)} \right\} < \frac{\alpha}{\beta} < \frac{9 \left(1-2\gamma\right) \left( 1- \gamma \right)}{2 \sqrt{3} \gamma^2}.
\end{equation}
Then the system \eqref{e:balanced} has the unique stationary solution $(x^*,y^*)=(0,0)$.
\end{cor}

Corollary~\ref{cor:nonexistence} leads directly to our main result, Theorem~\ref{e:main}, analytical and numerical regions are depicted in Fig.~\ref{fig:numerics}.
\paragraph{Proof of Thm.~\ref{t:main}} We apply substitutions~\eqref{e:parameters} to \eqref{e:alpha:omegas} to get \eqref{e:ineq:l1l2}. Local asymptotic stability of the origin follows from linearization. Global asymptotic stability is the consequence of the invariance of all sets $[0,k]^2$, $k\geq k_1$ and nonexistence of limit cycles.


\begin{figure}
\begin{minipage}[c]{0.62\linewidth}
\includegraphics[width=.49\linewidth]{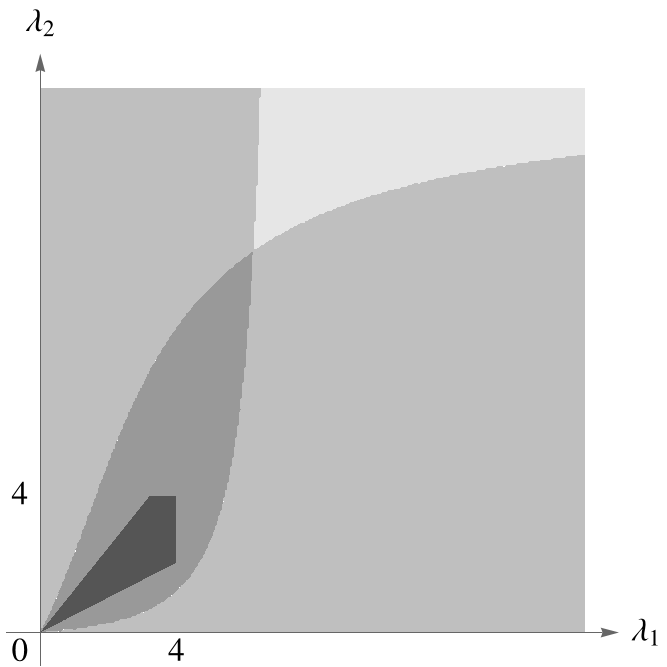}
\includegraphics[width=.49\linewidth]{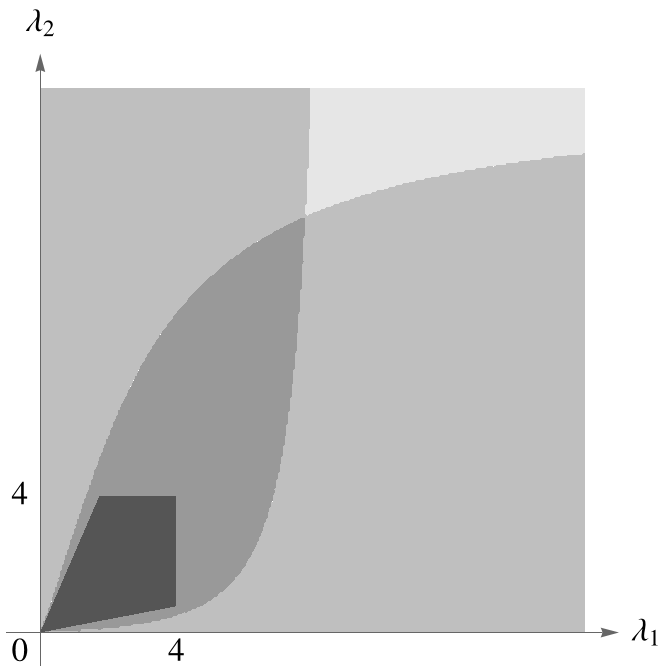}
\caption{Illustration of Theorem~\ref{t:main} with $D=1$, $k_1=1$, and $k_2=2/5$ (left), and $k_2=1/3$ (right). The darkest region correspond to the analytical result for the unique solution from Theorem~\ref{t:main}, the other gray regions (from the darkest to the lightest) correspond to the numerical regions with  unique, three, and five solutions.}\label{fig:numerics}
\end{minipage}
\hfill
\begin{minipage}[c]{0.3\linewidth}
\includegraphics[width=\linewidth]{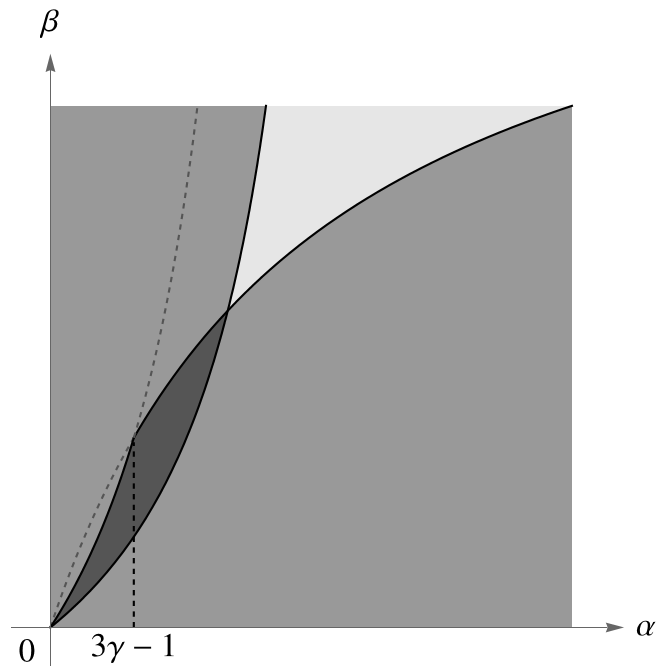}
\caption{Explicit regions for the unique solution, three, and five solutions (dark, medium, light gray, respectively) for the system~\eqref{e:balanced} with $\gamma = 11/25$ and the sawtooth reaction \eqref{e:sawtooth}, see Theorem~\ref{t:sawtooth}.}\label{fig:explicit}
\end{minipage}
\end{figure}

\section{Generalization $a\in(0,1)$}\label{sec:generalization}
Our approach could be replicated for a general viability constant $a\in(0,1)$ and the problem
\begin{equation}\label{e:a}
    \begin{cases}
    \displaystyle x_1' = D(x_2-x_1) + \lambda_1 x_1 \left(1-\frac{x_1}{k_1} \right) \left(\frac{x_1}{k_1}-a \right),\\
    \displaystyle x_2' = D(x_1-x_2) + \lambda_2 x_2 \left(1-\frac{x_2}{k_2} \right) \left(\frac{x_2}{k_2}-a \right),    
    \end{cases}
\end{equation}
For the sake of lucidity introduce bounds
\[
L(a,k_1,k_2)= \frac{ 2 k_1^2 \left((a+1)\left(a- \frac{1}{2} \right)(a-2) + \left(1+ a(a-1) \right)^{3/2}\right)}{27 (k_1 -k_2)}\cdot\max \left\lbrace \frac{1}{a^2(k_1-ak_2)}, \frac{1}{ak_1-k_2},
\right\rbrace.
\]
and
\[
U(a,k_1,k_2)= \frac{2(a k_1 - k_2)(k_1 -k_2)\left((a+1)\left(a-\frac{1}{2} \right)(a-2) - (1-a(1-a))^{3/2})\right)}{ (1-a) k_2^2}.
\]
The generalization of Thm.~\ref{t:main} can be obtained via the same path as in Section~\ref{sec:main}.
\begin{thm}\label{t:a}
Let $a\in(0,1)$ and assume that $\max\{\lambda_1,\lambda_2\}<
\frac{3D}{a^2-a+1}$, $k_2<a\cdot k_1$, and
\begin{equation}\label{e:ineq:l1l2:a}
    L(a,k_1,k_2) < \frac{\lambda_1}{\lambda_2} < U(a,k_1,k_2).
\end{equation}
Then the system \eqref{e:a} has the unique stationary solution $(x_1^*,y_1^*)=(0,0)$.
\end{thm}
Once we consider asymmetric viability constants $a_1,a_2$ in~\eqref{e:a}, we again arrive to inequalities of the type \eqref{e:ineq:l1l2:a}, in which the bounds become even more complicated and are thus omitted here.



On the contrary, if we consider a sawtooth reaction caricature \cite{McKean}, e.g.,
\begin{equation}\label{e:sawtooth}
f_{\mathrm{CAR}}(s) = \begin{cases}
    -s &\text{if }  s<1/4,\\
    s-1/2 &\text{if }  s \in \left(1/4, 3/4 \right),\\
    1-s &\text{if }  s > 3/4,
\end{cases}
\end{equation}
we can use the same procedures as above to determine explicit necessary and sufficient conditions on parameters $\alpha, \beta, \gamma $ which guarantee that the system \eqref{e:balanced} has the unique stationary solution, see Fig.~\ref{fig:explicit}.

\begin{thm}\label{t:sawtooth}
Let $\gamma \in \left(0,1/2\right).$ Then the system~\eqref{e:balanced} with the sawtooth reaction $f_{\mathrm{CAR}}$~\eqref{e:sawtooth} has the unique stationary solution $(x^*,y^*) = (0,0)$ if and only if
\[
\alpha (\beta +1) + \beta (4\gamma - 3) <0 \quad \text{ and } \quad \begin{cases}
\alpha \beta -3 \alpha + \beta <0 &\text{ if } \alpha \geq 3\gamma -1,\\
\alpha(2-3\gamma) + \beta \gamma (\alpha-1) >0 &\text{ if } \alpha < 3\gamma -1.
\end{cases}
\]
\end{thm}

\section{Discussion}\label{sec:discussion}

\paragraph{Global extinction} Similar results were obtained by numerical and simulation studies by  Amarasekare \cite{Amarasekare1998,Amarasekare1998b}, in which the role of mortality plays an essential role and Althagafi and Petrovskii \cite{Althagafi2021}, in which numerical persistence conditions are presented for a similar model, see Fig.~\ref{fig:numerics}.

\paragraph{Perfect mixing paradox} Our analysis is closely connected to a highly interesting recent critical discussion on perfect mixing paradox. Freedman and Waltman \cite{Freedman1977}  observed that for a model \eqref{e:main} with logistic nonlinearity, the nonzero equilibrium is not $k_1+k_2$ for $D\rightarrow\infty$. Arditi et al. \cite{Arditi2015} proved that the total limiting capacity is given by
\[
     N_T^* = k_1+k_2 +(k_1-k_2) \frac{\lambda_{1} k_2-\lambda_{2} k_1}{\lambda_{1} k_2+\lambda_{2} k_1},
\]
i.e., total capacity is higher if $(k_1-k_2)(\lambda_1-\lambda_2)>0$, lower if $(k_1-k_2)(\lambda_1-\lambda_2)<0$. Ramos-Jiliberto and Moisset de Expanés \cite{Ramos-Jiliberto2017} argued that this phenomenon is linked to unnatural form of diffusion term (present also in our model~\eqref{e:main}). Based on McPeek and Holt \cite{McPeek1992}, they argued that the proper diffusion between unbalanced patches should be described by the term $D(x_2/k_2-x_1/k_1)$ instead of $D(x_2-x_1)$ in the first equation in \eqref{e:main}, a fact corraborated by empirical evidence from \cite{Diffendorfer1998}. Arditi et al. \cite{Arditi2018} responded by a study in which the perfect mixing paradox is present also in a model with balanced dispersal. Naturally, in the spirit of this discussion the diffusion terms in \eqref{e:main}, although prevalent in the literature, could be seen as the main cause of the extinction result Thm.~\ref{e:main}.

\paragraph{Fragmentation} Study of spatially distributed metapopulations \cite{Hanski1998, Hanski2002} is motivated by both natural and anthro\-phogenic fragmentation of habitats \cite{Ewers2005}. The role of Allee effect has been extensively studied \cite{Courchamp2008}. Numerous experiments have been performed to see how the models compare to the different connectivities and number of patches in reality \cite{Damschen2019}. One of interesting question is the discussion whether a single large habitat or several small ones lead to higher diversity, the famous SLOSS debate \cite{Ovaskainen2002}.

\section*{Acknowledgments}
\noindent This work was supported by the Czech Science Foundation under Grant~GA22-18261S. The authors also kindly thank Vladim\'{\i}r \v{S}v\'{\i}gler and Jon\'{a}\v{s} Volek for their helpful remarks.

\end{document}